\documentclass[12pt]{amsart}
\usepackage{amssymb,verbatim,enumerate,ifthen}
\usepackage{esint}
\usepackage[mathscr]{eucal}
\usepackage[utf8]{inputenc}
\usepackage[T1]{fontenc}
\makeatletter
\@namedef{subjclassname@2010}{%
  \textup{2020} Mathematics Subject Classification}
\makeatother

\DeclareUnicodeCharacter{1EF3}{\`y}

\newtheorem{thm}{Theorem}

\newtheorem{prop}{Proposition}
\newtheorem{lem}{Lemma}

\newcommand\R{\mathbb{R}}
\newcommand\N{\mathbb{N}}
\newcommand\Hc{\mathscr{H}}

\newcommand{\M}{\mathscr{M}}

\author[P. Pasteczka]{Pawe\l{} Pasteczka}
\address{Institute of Mathematics \\ Pedagogical University of Krak\'ow \\ Podchor\k{a}\.zych str. 2, 30-084 Krak\'ow, Poland}
\email{pawel.pasteczka@up.krakow.pl}

\subjclass[2010]{26D15, 26E60}
    
\keywords{Hardy inequality, weak-Hardy inequality, continuouity, selfmappings of $\ell_1$}

\newcommand{\operator}[1]{\mathop{\vphantom{\sum}\mathchoice
{\vcenter{\hbox{\LARGE $#1$}}}
{\vcenter{\hbox{\Large $#1$}}}{#1}{#1}}\displaylimits}

\def\Mst_#1^#2{\operator{\mathscr{M}_{\mbox{\scriptsize$\#$}}\!\!}_{#1}^{#2}\,\,}

\def\eq#1{{\rm(\ref{#1})}}
\def\Eq#1#2{\ifthenelse{\equal{#1}{*}}
  {\begin{equation*}\begin{aligned}[]#2\end{aligned}\end{equation*}}
  {\begin{equation}\begin{aligned}[]\label{#1}#2\end{aligned}\end{equation}}}

\title{On negative results concerning weak-Hardy means}
\begin{document}

\begin{abstract}
We establish the test which allows to show that a mean does not admit a weak-Hardy property. As a result we prove that Hardy and weak-Hardy properties are equivalent in the class of  homogeneous, symmetric, repetition invariant, and Jensen concave mean on $\R_+$.

More precisely, for every mean $\M \colon \bigcup_{n=1}^\infty \R_+^n \to \R$ as above, the inequality $$\M(a_1)+\M(a_1,a_2)+\dots<\infty$$
holds for all $a \in \ell^1(\R_+)$ if and only if there exists a positive, real constant $C$ (depending only on $\M$) such that
$$\M(a_1)+\M(a_1,a_2)+\dots<C \cdot (a_1+a_2+\cdots)$$
for every sequence $a \in \ell^1(\R_+)$.
\end{abstract}

\maketitle
\section{Introduction}
The history of the Hardy property stared in 1920, when it was shown for all power means below the arithmetic one (see Hardy~\cite{Har20}). Later it was improved by the number of authors, which is described in details in surveys by Pe\v{c}ari\'c--Stolarsky \cite{PecSto01}, Duncan--McGregor \cite{DunMcg03}, and in a book of Kufner--Maligranda--Persson \cite{KufMalPer07}.

Recall that a mean $\M$ on an interval $I \subset [0,+\infty)$ (that is a function $\M \colon  \bigcup_{n=1}^\infty I^n \to I$ such that $\min(v)\le \M(v)\le \max(v)$ for every vector $v$ in its domain)
is said to be a \emph{Hardy mean} if there exists a finite constant $C$ such that
\Eq{E:defH}{
\sum_{n=1}^\infty \M(a_1,\dots,a_n) \le C \sum_{n=1}^\infty a_n \quad \text{ for all }\quad a \in \ell_1(I).
}
This definition appears first time in P\'ales-Persson \cite{PalPer04}, however there are a number of earlier results which can be expressed in this framework (including the Hardy's one).
The smallest extended real number $C$ satisfying \eq{E:defH} is called the \emph{Hardy constant of $\M$} and denoted by $\Hc(\M)$. In this sense Hardy means are precisely the ones with a finite Hardy constant.

A slight modification of this definition was proposed recently by the author in \cite{Pas20b}. Namely, we say that $\M$ is a \emph{weak-Hardy mean} if 
\Eq{*}{
\sum_{n=1}^\infty \M(a_1,\dots,a_n) < +\infty \quad \text{ for all }\quad a \in \ell_1(I).
}
It is obvious that every Hardy mean is also weak-Hardy, nevertheless the converse implication is not valid. Indeed, one can find a broad class of weak-Hardy, quasiarithmetic means which are not Hardy means (see \cite{Pas20b} for details). 

An important test which allows to show that a mean has no Hardy property was proved by the author in \cite{Pas15c} (see Proposition~\ref{prop:NegHar} below). In \cite{Pas20b} there was a test which allows to verify that a homogeneous monotone, and repetition invariant mean in not a weak-Hardy mean (see Proposition~\ref{prop:NegWHar} below), but it was in completely different spirit than the one in \cite{Pas15c}. 

The aim of this paper is to prove an analogue of the result from \cite{Pas15c} for the weak-Hardy property. It turns out that in this case we need to additionally claim that a mean is homogeneous and monotone. As a result we prove that the Hardy and the weak-Hardy property coincide for means which are concave, homogeneous, and repetition invariant.

\subsection*{Properties of means}
Based on \cite{PalPas16}, we say that $\M$ is \emph{symmetric} (resp. \emph{concave}) if for all 
$n\in\N$ the $n$-variable restriction $\M|_{I^n}$ is symmetric (resp. concave). If $I=\R_+$, we can analogously define the notion of homogeneity of $\M$. \emph{Monotonicity} of mean is associated with its nondecreasingness in each variable. Finally, the mean $\M$ is called \emph{repetition invariant} if for all $n,m\in\N$
and $(v_1,\dots,v_n)\in I^n$ the following identity is satisfied
\Eq{*}{
  \M(\underbrace{v_1,\dots,v_1}_{m\text{-times}},\dots,\underbrace{v_n,\dots,v_n}_{m\text{-times}})
   =\M(v_1,\dots,v_n).
}

\subsection*{State of arts}

Following the idea of \cite{Pas20b}, we say that a sequence $(a_n)_{n=1}^\infty$ of positive numbers is \emph{nearly increasing} if there exists $\varepsilon>0$ such that for every $m,\,n \in\N$ with $m \le n$ we have $\varepsilon a_m \le a_n$. In what follows, let us recall few results related to this topic.

\begin{prop}[\cite{Pas15c}, Theorem~1.1]\label{prop:NegHar}
Let $I \subset \R_+$ be an interval with $\inf I=0$ and $\M \colon \bigcup_{n=1}^\infty I^n \to I$ be a mean. If there exists a sequence $(a_n)_{n=1}^\infty$ of numbers in $I$ such that 
\Eq{E:condThm}{
\sum_{n=1}^\infty a_n=+\infty \quad\text{ and }\quad\lim_{n \to \infty} a_n^{-1} \M(a_1,\dots,a_n) = +\infty,
}
then $\M$ is not a Hardy mean.
\end{prop}

\begin{prop}[\cite{Pas20b}, Theorem~2.2]\label{prop:NegWHar}
Let $\M$ be a homogeneous and monotone mean defined on $\R_+$. If there exists a sequence $(a_n)$ of positive numbers such that
 \begin{enumerate}[i)]
  \item $\sum_{n=1}^\infty a_n=+\infty$, 
  \item a sequence $(a_n^{-1}\M(a_1,\dots,a_n))_{n=1}^\infty$ is nearly increasing and divergent,
  \item $\sum_{n=1}^\infty a_n^{1+s} \big(\M(a_1,\dots,a_n)\big)^{-s}$ is finite for some $s \in \R_+$,
 \end{enumerate}
then $\M$ is not a weak-Hardy mean.
\end{prop}

\begin{prop}[\cite{Pas20b}, Corollary~2.3]
Let $\M$ be a homogeneous, monotone, and repetition invariant mean. If there exist $C,\,D\in \R_+$ and $n_0 \in \N$ such that 
\Eq{*}{
\M(a_1,\dots,a_n) \ge \frac{C (\ln n)^D}n \qquad \text{for all }n \ge n_0,
}
then $\M$ is not a weak-Hardy mean.
\end{prop}

\section{Results}
We are going to generalize Proposition~\ref{prop:NegWHar} in the spirit of Proposition~\ref{prop:NegHar}. Observe that, as for now, we have two additional assumptions in Proposition~\ref{prop:NegWHar} (the second and the third one). We do relax (or even omit) them. 

Let us start with a purely technical lemma, which shows the existence of a sequence which in some sense breaks the convergence property.

\begin{lem}\label{lem:1}
For every sequences  $(a_n)_{n=1}^\infty$, $(c_n)_{n=1}^\infty$ of positive reals such that
\begin{itemize}
    \item $(a_n)_{n=1}^\infty$ is bounded;
    \item $c_n \to +\infty$;
    \item $\sum_{n=1}^\infty a_n = +\infty$.
\end{itemize}
Then there exists a strictly decreasing sequence $(r_n)_{n=1}^\infty$ such that 
\Eq{E:lem1}{
\sum_{n=1}^\infty a_nr_n<+\infty\quad\text{ and }\quad \sum_{n=1}^\infty a_nc_nr_n=+\infty.
}
\end{lem}
\begin{proof}
Observe that if \eq{E:lem1} holds for some nonincreasing sequence $(r_n)$ then it also valid with $r_n$ replaced  by $(1+\tfrac1n) r_n$, which is already strictly decreasing. Thus we can show that there exists a nonincreasing sequence $(r_n)_{n=1}^\infty$ satisfying \eq{E:lem1}. Let us split our proof to two cases.

\medskip
\noindent {\sc Case 1.} If $\inf a_n>0$ then $\tfrac{\sup a_n}{\inf a_n}$ is a finite number. Therefore coefficients $a_n$ on the left-hand-sides of \eq{E:lem1} does not affect their convergence (or divergence). Consequently one can assume without loss of generality that $a_n \equiv 1$.

Since $c_n \to +\infty$, there exists a strictly increasing sequence $(n_k)_{k=1}^\infty$ of natural numbers such that $c_n>k$ for all $n\ge n_k$.
Moreover one can additionally define $n_0:=0$ and claim the mapping $k \mapsto n_{k+1}-n_k$ to be increasing.

Now define $(r_n)_{n=1}^\infty$ by  
\Eq{*}{
r_n:=
\frac{1}{(k+1)^2(n_{k+1}-n_k)} & \text{ for }k \in\{0,1,\dots\}\text{ and }n \in \{n_k+1,\dots,n_{k+1}\}.
}
Then $(r_n)$ is nonincreasing, and
\Eq{*}{
\sum_{n=1}^\infty r_n&=
\sum_{k=0}^\infty \sum_{n=n_k+1}^{n_{k+1}} r_n=\sum_{k=0}^\infty \frac{1}{(k+1)^2}<+\infty;\\
\sum_{n=1}^\infty c_nr_n
 &=\sum_{k=0}^\infty \sum_{n=n_k+1}^{n_{k+1}} c_nr_n
\ge\sum_{k=0}^\infty k \sum_{n=n_k+1}^{n_{k+1}} r_n
  =\sum_{k=0}^\infty \frac{k}{(k+1)^2}=+\infty,
}
which completes the prove in the first case.

\medskip
\noindent {\sc Case 2.} In the case $\inf a_n=0$ define the sequence $(q_n)_{n=1}^\infty$ inductively in the following way: $q_1=0$ and $q_{k+1}$ is the smallest natural number above $q_k$ such that $A_k:=\sum_{n=q_k+1}^{q_{k+1}} a_n \ge 1$. 
Denote briefly $Q_k:=\{q_{k}+1,\dots,q_{k+1}\}$, and $C_k:=\inf_{n \in Q_k} c_n$. As $q_k \to \infty$ and $c_n \to \infty$,we have $C_n \to \infty$.

At this stage, in view of the first case of this statement, there exists a nonincreasing sequence $(R_k)_{k=1}^\infty$ such that 
\Eq{*}{
\sum_{k=1}^\infty R_k A_k <\infty \text{ and }\sum_{k=1}^\infty C_kA_k R_k = +\infty.
}
Using this sequence we define $(r_n)_{n=1}^\infty$ by the property $r_n:=R_k$ for all $n \in Q_k$. Then
\Eq{*}{
\sum_{n \in Q_k} a_n r_n= \sum_{n \in Q_k} a_n R_k=A_kR_k.
}
Finally we have 
\Eq{*}{
\sum_{n=1}^\infty a_n r_n &= \sum_{k=1}^\infty \sum_{n \in Q_k} a_n r_n=\sum_{k=1}^\infty A_kR_k < \infty;\\
\sum_{n=1}^\infty a_n c_n r_n &= \sum_{k=1}^\infty \sum_{n \in Q_k} a_n c_n r_n
\ge \sum_{k=1}^\infty C_k \sum_{n \in Q_k} a_n  r_n =\sum_{k=1}^\infty C_k A_kR_k=+\infty.
}
It completes the proof.
\end{proof}

Now we proceed to our main theorem.
\begin{thm}\label{thm:main}
Let $\M \colon \bigcup_{n=1}^\infty \R_+^n \to \R_+$ be a homogeneous and monotone mean. If there exists a sequence 
$(a_n)_{n=1}^\infty$ of positive reals satisfying \eq{E:condThm} then $\M$ is not a weak-Hardy mean.
\end{thm}
\begin{proof}
Fix the sequence $(a_n)$ such that $\sum_{n=1}^\infty a_n=+\infty$ and the related sequence of ratios 
\Eq{*}{
c_n:=a_n^{-1} \M(a_1,\dots,a_n)
}
is divergent. In particular there exists $n_0 \in \N$ such that $c_n >1$ for all $n > n_0$. However
$c_n>1$ yields
\Eq{*}{
a_n < c_na_n=\M(a_1,\dots,a_n) \le \max(a_1,\dots,a_n).
}
Thus $\max(a_1,\dots,a_{n-1})=\max(a_1,\dots,a_{n})$ for all $n > n_0$. 
Consequently $(a_n)$ is bounded from above by $\max(a_1,\dots,a_{n_0})$. Thus, by Lemma~\ref{lem:1}, there exists a decreasing sequence $(r_n)_{n=1}^\infty$ such that \eq{E:lem1} holds.

Now we have
\Eq{*}{
\M(a_1r_1,\dots,a_nr_n)&=r_n\M\Big(a_1\frac{r_1}{r_n},\dots,a_n\frac{r_n}{r_n}\Big) \\ 
&\ge r_n\M(a_1,\dots,a_n)=c_na_nr_n.
}

Consequently
\Eq{*}{
\sum_{n=1}^\infty \M(a_1r_1,\dots,a_nr_n) =\sum_{n=1}^\infty c_n^* a_n r_n \ge\sum_{n=1}^\infty c_n a_n r_n =+\infty,
}
which, in view of the inequality $\sum_{n=1}^\infty a_nr_n<+\infty$, shows that $\M$ is not a weak-Hardy mean.
\end{proof}

\section{Applications}
Recall that by \cite[Corollary~3.5]{PalPas16} for every homogeneous, monotone, symmetric, repetition invariant, and Jensen concave mean $\M$ on $\R_+$ we have
\Eq{*}{
  \Hc(\M)= \lim_{n \to \infty} n \cdot \M\big(1,\tfrac 12,\ldots,\tfrac 1n \big).
}
Moreover, under these assumptions, the limit on the right-hand-side always exists (possibly it is infinite). Therefore in this family two mentioned Hardy-type properties coincide. More precisely one can show the following result.
\begin{prop}
Let $\M$ by a homogeneous, symmetric, repetition invariant, and Jensen concave mean on $\R_+$. Then $\M$ is a Hardy mean if and only if it is a weak-Hardy mean.
\end{prop}
\begin{proof}
Obviously every Hardy mean is a weak-Hardy mean too. To show the converse implication recall that by \cite[Lemma~2]{Pas2104}, since $\M$ is concave, is also monotone.

If $\M$ is a weak-Hardy mean then, applying Theorem~\ref{thm:main} to the test sequence $a_n:=\tfrac1n$, we obtain that \eq{E:condThm} does not hold. However $\sum_{n=1}^\infty a_n=\infty$ and the limit 
$C:=\lim_{n \to \infty} n \cdot \M\big(1,\tfrac 12,\ldots,\tfrac 1n \big)$ exists. Thus the only possibility is that $C$ is finite. 

On the other hand, by \cite[Corollary~3.5]{PalPas16}, we obtain $\Hc(\M)=C<+\infty$, thus $\M$ has a finite Hardy constant, i.e. it is a Hardy mean.
\end{proof}

\end{document}